\documentclass{article}
\usepackage[utf8]{inputenc} 
\usepackage{geometry} 
\usepackage{graphicx} 

\usepackage{algorithm2e}
\usepackage{algorithmic}

\usepackage{booktabs} 
\usepackage{array} 
\usepackage{paralist}
\usepackage{verbatim}
\usepackage{subfig} 
\usepackage{fancyhdr}
\usepackage{amsmath}
\usepackage{amstext}
\usepackage{amssymb}
\usepackage{amsfonts}
\usepackage{amsthm}

\usepackage{multirow}
\usepackage{setspace}

\usepackage{tikz}

\newcommand{\graph}[1]{#1}

\newcommand{\vgraphone}[1]{V_1(#1)}
\newcommand{\vgraphtwo}[1]{V_2(#1)}
\newcommand{\edgesgraph}[1]{E(#1)}

\newcommand{\ngraph}[1]{\mathcal{N}(#1)}
\newcommand{\ngraphT}[1]{\mathcal{N}_T(#1)}

\newcommand{\affinegraphone}[1]{\Gamma_1}
\newcommand{\affinegrapha}[1]{\Gamma_2}

\newcommand{\vectorname}[1]{\textbf{#1}}

\newcommand{\ff}[1]{{\mathbb F}_{#1}}

\newcommand{\projectivespace}[2]{{\mathbb P}^{#1}(#2)}

\newcommand{\lines}{\mathcal{L}}

\newcommand{\grassmann}[2]{\mathcal{G}_{#1, #2}}

\newcommand{\grassmanncode}[2]{\mathcal{C}(#1, #2)}

\newcommand{\schubertcodealpha}[3]{\mathcal{C^{#3}}(#1, #2)}

\newcommand{\cl}[3]{cl_{#1, #2}(#3)}

\newtheorem{thm}{Theorem}[section]

\newtheorem{defn}[thm]{Definition}

\newtheorem{lem}[thm]{Lemma}
\newtheorem{cor}[thm]{Corollary}

\newtheorem{prop}{Proposition}

\title{The structure of dual Schubert union codes}
\author{Fernando L. Pi\~nero \\ University of Puerto Rico -- Ponce}

\begin{document}
\maketitle

\begin{abstract}

In this article we prove that Schubert union codes are Tanner codes constructed from the point--line incidence geometry inherited from the Grassmannian. Our proof is based on an iterative encoding algorithm for Tanner codes.  This encoder determines the entries of a codeword of a Tanner code from the entries in a given subset of its positions. As a result, we find sufficient conditions on the initial positions such that a codeword is determined from the component codes only. This algorithm has linear complexity on the code length. We also use this encoder to determine the minimum distance of Schubert union codes in terms of the minimum distance of the Schubert varieties contained therein.
\end{abstract}

\section{Introduction}

The Grassmann variety may be defined as the set of all subspaces of a finite dimensional vector space $V$ with a fixed dimension. We shall focus on $V = \ff{q}^m$ and consider the Grassmannian of subspaces of dimension $\ell$. The Grassmannian over other fields is also important. The linear codes  from the Grassmann variety, Schubert variety and Schubert unions are used to understand the projective systems of the associated varieties \cite{NTV07}. For example, the generalized Hamming weights hold information about these varieties. Tanner introduced Tanner codes in \cite{T81} as a way of building a long, complex code from a shorter, simpler code and a bipartite graph. We aim to study Schubert union codes as Tanner codes. This construction reflects the way Schubert unions are constructed from their related finite incidence structures.

We lay some ground work on encoding Tanner codes with irreversible $k$--threshold processes. We give conditions on the positions which must be encoded correctly in order to determine the remaining parity check bits using only the graph and the component code. Schubert union codes may be encoded iteratively using only an encoder of the doubly extended Reed-Solomon code. As a corollary, the dual codes of Schubert union codes are generated by their minimum weight codewords. 

We conclude the article with some lower and upper bounds on the minimum distance of Grassmann codes using the eigenvalues of the Tanner graph. This offers an alternative method to understand the minimum distance and Generalized Hamming weight spectrum of Grassmann codes. We hope this approach could prove useful for determining the Generalized Hamming Weights of Grassmann codes and Schubert codes. First, we recall some basic concepts from coding theory.

Let $A$ be a finite set and $q$ a prime power. The elements of $\ff{q}^A$ are considered as functions from $A$ to $\ff{q}$. A function $f: A \rightarrow \ff{q}$ is represented as the vector $(f_a)_{a \in A}$ where $f_a = f(a)$. Usually $A = \{1,2, \ldots, n\}$, but we may use any finite set to index the coordinates.  By a \emph{code} in $\ff{q}^A$, we shall mean an $\ff{q}$--linear subspace of $\ff{q}^A$. 

Projecting a code onto some coordinates is a fundamental operation to make short codes from longer codes.  If $C$ is a code in $\ff{q}^A$ and $\phi: B \rightarrow B'$ is a bijection of $B \subseteq A$ onto $B'$, then $\phi$ induces a map of $C^B$ into a code on $B'$. We also denote this map by $\phi$. 

\begin{defn}\cite{MS77}
Let $C$ be a code in $\ff{q}^A$. Let $B \subseteq A$.  The \emph{projection of $C$ on $B$} is the code in $\ff{q}^B$ obtained by projecting $C$ onto the coordinates given by $B$, that is: 
$$C^B := \{ (c_i)_{i \in B} \ | \ (c_i)_{i \in A} \in C \} .$$
\end{defn}

Note that $C^B$ is a code of length $\#B$. To project a code $C$ onto $B$ is to discard or delete the coordinates in $A\setminus B$. When $\dim C^B  = \dim C$, the code $C^B$ is also the \emph{puncturing of $C$ at $A \setminus B$}.

\begin{defn}\cite{MS77}
Let $C$ be a code in $\ff{q}^A$. For $c \in C$ we define the \emph{support of $c$} as $$supp(c) := \{  a \in A \ | \ c_a \neq 0 \}.$$ 
\end{defn}

\begin{defn}\cite{MS77}
Let $C$ be a code in $\ff{q}^A$. Let $B \subseteq A$. We say $B$ is an \emph{information set} for $C$ if $\# B = \dim C$ and $C^B = \ff{q}^B$. Suppose $B' \subseteq A$ contains an information set $B$. A \emph{lengthening of $C^{B}$} is to determine in any way $m \in C \subseteq \ff{q}^{A}$ from its image $m^{B'} \in C^{B'}$.  If $B$ is an information set, lengthening $C^{B}$ to $C$ is a \emph{systematic encoder}.\end{defn}

Projecting $C$ onto $B$ is a linear mapping from $C$ to $C^B.$ When $B$ contains an information set, this mapping is a linear isomorphism. A code is lengthened by adding a parity check bit. Lengthening increases the length while keeping a reasonable bound on the dimension and minimum distance. Lengthening a code $C^B$ up to $C^A$ is the same as determining the parity check bits in $A \setminus B$.

\section{Tanner codes}

Tanner codes, introduced in \cite[Section II]{T81}, are a class of codes constructed from a bipartite graph and a shorter code. Any code may be described as a Tanner code. The main feature of Tanner codes is their iterative decoding algorithm, which means a fixed proportion of errors can be corrected easily. 

\begin{defn}
A bipartite graph $\graph{G}$ is a triple $(\vgraphone{G}, \vgraphtwo{G}, \edgesgraph{G})$ where $\vgraphone{G}$ and $\vgraphtwo{G}$ are finite sets and $\edgesgraph{G} \subseteq \vgraphone{G} \times \vgraphtwo{G}$. The elements of $\vgraphone{G}$ and $\vgraphtwo{G}$ are called \emph{vertices.}  For a vertex $u \in \vgraphtwo{G}$ we define the \emph{neighborhood} of $u$ as $$\ngraph{u} := \{  v \in \vgraphone{G} \  | \ (v,u) \in \edgesgraph{G} \}.$$
\end{defn}

A bipartite graph $\graph{G} = (\vgraphone{G}, \vgraphtwo{G}, \edgesgraph{G})$ may be also represented by identifying $\vgraphtwo{G}$ with the collection $\{ \ngraph{u} \subseteq | \ u \in V_2(G)  \}$ of subsets of $V_1(G)$. Likewise, any such collection of subsets of $\vgraphone{G}$ determine an unique bipartite graph. This view is closer to incidence geometries.  When $V_2(G)$ is represented by a collection of subsets of $V_1(G)$, the edge set $\edgesgraph{G}$ is defined by inclusion. We impose a right regularity condition to simplify our definitions. 

\begin{defn}
Let $\graph{G} = (\vgraphone{G}, \vgraphtwo{G}, \edgesgraph{G})$ be a bipartite graph.  Let $n' \leq \#V_1(G)$ be a positive integer. The graph $G$ is an \emph{$n'$-right regular bipartite graph} if $$\#\ngraph{u} = n' \makebox{ for each } u \in V_2(G).$$  \end{defn}

\begin{defn}\cite{T81} Let $n'$ be an integer and let $\mathcal{N}'$ be a set such that $ \# N' = n'$. Suppose $G$ is an $n'$-right regular bipartite graph. Let $C'$ be a code in $\ff{q}^{\mathcal{N}'}$. We say $C$ is a \emph{Tanner code with component code $C'$ and associated bipartite graph $G$} if $C$ is a code in $\ff{q}^{\vgraphone{G}}$  such that for each $u \in \vgraphtwo{G}$ there exists $C'_u$, a code equivalent to $C'$, and $\phi_u$, a bijection between $\ngraph{u}$ and $\mathcal{N}'$, such that $$\phi_u(C^{\ngraph{u}}) \subseteq C'_u.$$ To emphasize the role of $G$ and $C'$, the code $C$ is usually denoted by $(G,C')$.
\end{defn}

A linear code $C$ is defined in terms of parity check equations. Instead of having parity check equations on all of $V_1(G)$ a Tanner code uses only some short parity check equations defined in terms of the shorter, simpler component code, $C'$ on the subsets $\ngraph{u} \subseteq \vgraphone{G}$ for $u \in \vgraphtwo{G}$. The vertices in $\vgraphone{G}$ are known as \emph{variable nodes}. The vertices in $\vgraphtwo{G}$ are \emph{constraint nodes} because they represent parity check equations $(G,C')$ must satisfy. Tanner codes are also known as \emph{generalized LDPC codes}. 

\begin{defn}Let $n'$ be an integer and let $\mathcal{N}'$ be a set of size $n'$.  Suppose $G$ is an $n'$-right regular bipartite graph. Let $C'$ be a code of length $n'$ in $\ff{q}^{\mathcal{N}'}.$ Let $\phi = (\phi_u) $ be an $\# V_2(G)$-tuple where each $\phi_u$ is a bijective map from $\ngraph{u}$ to $\mathcal{N}'.$ In addition,  let  $\mathcal{C}= (C'_u) $ be an $\# V_2(G)$-tuple where each $C'_u$ is equivalent to $C'$.  For $u \in V_2(G)$, let $D_u \leq \ff{q}^{V_1(G)}$  be the code which satisfies:  $$D_u^{\ngraph{u}} = \phi_u^{-1}(C'_u)^\perp$$ and $$supp(d) \subseteq \ngraph{u}$$ for each $d \in D_u$.

We define the maximal Tanner code for $\phi$ and $\mathcal{C}$ as $$(G,C')_{\phi, \mathcal{C}} = D^\perp$$  $$\makebox{ where } D = span(\{ D_u \ | \ u \in V_2(G)\})^{\perp}.$$\label{def:MaxTannercode}\end{defn}

Usually one focuses on $G$ and $C$ when working with Tanner codes. The graph $G$, the component code $C'$ and the codes $C'_u$, and the permutations $\phi_u$ also play an important role in the Tanner code's performance. Different permutations or equivalent codes may change the Tanner code. Now we prove that $(G,C')_{\phi, \mathcal{C}}$ is the largest code which is a Tanner code for $\phi_u$ and the codes $C'_u$.
\begin{lem}
 Let $n'$ be an integer and let $\mathcal{N}'$ be a set of size $n'$. Suppose $G$ is an $n'$-right regular bipartite graph. Let $C'$ be a code of length $n'$ in $\ff{q}^{\mathcal{N}'}.$ Suppose $C = (G, C')$ is a Tanner code for the codes $C'_u$ and the bijections $\phi_u$. If $\phi = (\phi_u) $ and $\mathcal{C}= (C'_u)$, then $C \subseteq (G,C')_{\phi, \mathcal{C}}.$
\label{lem:MaxTannerthm}
\end{lem}

\begin{proof}
The containment $\phi_u^{-1}(C^{\ngraph{u}}) \subseteq C'_u$ implies $(C'_u)^\perp \subseteq \phi_u^{-1}(C^{\ngraph{u}})^\perp.$ For $u \in V_2(G)$ consider the code $D_u \leq \ff{q}^{V_1(G)}$ where the code $D_u^{\ngraph{u}} = \phi_u^{-1}(C'_u)^\perp$ and $supp(d) \subseteq \ngraph{u}$ for $d \in D_u$.  Therefore, $D_u \subseteq C^\perp$ for each $u \in V_2(G)$ which implies $C \subseteq D^\perp = (G,C')_{\phi, \mathcal{C}}$.\end{proof}

Any code is a Tanner code. For example, when the length of a cyclic code, $C$,  is coprime to the field characteristic, the parity checks for $C$ are generated by the cyclic shifts of a single vector, say $h = (h_0, h_1, \ldots, h_{n-1})$. Define $G$ as the bipartite graph where $V_1(G) = V_2(G)  = \mathbb{Z}_n$ and $(a,b) \in E(G)$  if and only if $a-b \in supp(h)$. Define $h'$ as the projection of $h$ onto $supp(h)$. Let $C' \leq \ff{q}^{supp(h)}$ where $C' = (h')^\perp$. Note that for each $b \in V_2$,  $\ngraph{b} = \{a+b \ | \ a \in supp(h)  \}$. Thus for $b \in V_2$ we consider the bijection $\phi_{b}$ from $\ngraph{u}$ onto $supp(h)$ such that $\phi_b(a+b) = a$. With this description the cyclic shifts of $h$ are the parity check equations of $(G, C')$. Therefore, $(G,C')  = C$.

\section{Irreversible $k$--threshold processes and $k$--forcing sets}

In this section, we consider irreversible $k$--threshold processes in a bipartite graph. Our aim is to simulate encoding a Tanner code as such a process, and determine the consequences this may have for Tanner codes.  Irreversible $k$--threshold processes are defined as follows.

\begin{defn}[Irreversible $k$--threshold process] \cite{DR09}.

Suppose $G = (V,E)$ is a graph. At time $t=0$, each vertex of $G$ has one of two states, either $0$ or $1$.  When the time $t$ increases from $i$ to $i+1$,  a vertex switches from state $0$ to state $1$ when at least $k$ of its neighbors are in state $1$. Once a vertex is in state $1$, it shall remain in state $1$.

\end{defn}

Originally irreversible $k$--threshold processes are defined for a graph $G$, see \cite[Section 1.1]{DR09}. State $0$ can represent uninfected people, and state $1$ may represent infected people. We are particularly interested when state $0$ represents an unencoded bit and state $1$ represents an encoded bit. We state a bipartite version of this definition which turns out to be useful in encoding Tanner codes.

\begin{defn}[Irreversible $k$--threshold process for bipartite graphs].

Suppose the graph $G = (\vgraphone{G}, \vgraphtwo{G},E(G) )$ is an $n'$--right regular bipartite graph. For $t=0$,  each vertex in  $\vgraphone{G}$ has one of two states, either $0$ or $1$.   When the time $t$ increases from $i$ to $i+1$ and these exists $u \in \vgraphtwo{G}$ with at least $k$ neighbors with state $1$, then all $v \in \ngraph{u}$ also switch to state $1$. Once a vertex $v \in \vgraphone{G}$ is in state $1$ it will always remain in state $1$.

\end{defn}

\begin{algorithm}
\begin{algorithmic}
\STATE Input: $G$ a bipartite $n'$-right regular graph, $1 \leq k \leq n'$ and $S \subseteq V_1(G)$.
\STATE Initialize $Z := S.$
\STATE Initialize $t := 0$
\WHILE {$\exists u \in V_2(G) \ : \ k \leq \# (\ngraph{u} \cap Z)  < n'$}
\STATE Initialize $t := t+1$
\STATE Set $Z := Z \cup \ngraph{u}$
\ENDWHILE
\STATE Return $Z$
\end{algorithmic}
\caption{ Irreversible $k$--threshold process for bipartite graphs}
\label{alg:1}
\end{algorithm}

The set $Z$ in Algorithm \ref{alg:1} represents the vertices of $V_1(G)$ in state $1$. Once a vertex in $u \in V_2(G)$ has at least $k$ neighbors in $Z$, then all  $v \in \ngraph{u}$ are added to $Z$.  As $\vgraphtwo{G}$ is finite, this process ends in a finite number of steps.

\begin{lem}
Let $G$ be an $n'$-right regular bipartite graph, $k \leq n'$ and $S \subseteq \vgraphone{G}$. The output of Algorithm \ref{alg:1} does not depend on the order the vertices are added to the initial set $S$.
\end{lem}
\begin{proof} 

Suppose $Z$ and $Z'$ are two possible outputs of Algorithm \ref{alg:1} when the graph $G$, the set $S$ and $k$ are given as input. For $i =1,\ldots , a$, suppose $\ngraph{u_i}$ is added to $Z$ at time $t=i$. For each $i = 0, 1,\ldots , a$, define $$Z_i = S \cup \ngraph{u_1} \cup \ngraph{u_2} \cup \ldots \cup \ngraph{u_i}.$$ Likewise suppose that the neighborhood $\ngraph{t_j}$ is added to $Z'$ at time $j$. For $j = 0, 1,\ldots , b$, define $$Z'_j = S \cup \ngraph{t_1} \cup \ngraph{t_2} \cup \ldots \cup \ngraph{t_j}.$$ Note that $Z = Z_a$ and $Z' = Z'_b$. 

The statement of Algorithm \ref{alg:1} implies $k \leq \# (\ngraph{u_{i+1}} \cap Z_{i}) < n'$ for $i=0,1,\ldots, a-1$ and $k \leq \# (\ngraph{t_{j+1}} \cap Z'_{j}) < n'$ for $j=0,1,\ldots, b-1$.

We claim $Z \subseteq Z'$. Otherwise, if $Z \not\subseteq Z'$, let $i$ be minimal such that $Z_{i} \subseteq Z'$, but $Z_{i+1} \not\subseteq Z'$. Since $k \leq \#(Z_i \cap \ngraph{u_{i+1}}) < n'$ the elements in $\ngraph{u_{i+1}}$ will be added to some $Z'_j$ at some point in Algorithm \ref{alg:1}, therefore, $\ Z_{i+1} \subseteq Z'$ follows. This contradicts the choice of $i$. Thus, $Z \subseteq Z'$.  Equality follows after exchanging the roles of $Z$ and $Z'$ in the previous argument. \end{proof}

\begin{defn}

The output of Algorithm \ref{alg:1} for the graph $G$, the set $S$ and $k$ are given as input is denoted by $\cl{G}{k}{S}$. The set $\cl{G}{k}{S}$ is known as the \emph{$k$--closure of $S$ with respect to $G$}. A set of the form $\cl{G}{k}{S}$ is said to be \emph{$k$--closed w.r.t. $G$}.

\end{defn}

\begin{lem}
Let $G$ be an $n'$-right regular bipartite graph. Let $k \leq n'$.  The set $T \subseteq \vgraphone{G}$ is a $k$--closed set if and only if  for all $u \in \vgraphtwo{G}$ either $\#(T \cap \ngraph{u})  < k \makebox{ or } \ngraph{u} \subseteq T.$\label{lem:closedsetintersectionnum}
\end{lem}

\begin{lem}

Let $G$ be an $n'$-right regular bipartite graph, $k \leq n'$ and $S \subseteq \vgraphone{G}$ If $t \in V_2(G)$ is such that $k \leq \ngraph{t} \cap S < n'$, then $\cl{G}{k}{S} = \cl{G}{k}{S\cup \ngraph{t}}$.
\label{lem:closedsubsets}
\end{lem}

 In the classical definition of an induced subgraph, the induced subgraph $T \subseteq V_1(G)$  would have no edges. For a bipartite graph $G$, and $T \subseteq V_1(G)$  we shall define the induced graph by $T$ as follows.
\begin{defn}
Let $G$ be an $n'$-right regular bipartite graph. Let $T \subseteq V_1(G)$. Define the vertex set $V_2(G)^T := \{ u \in V_2(G) \ | \ \ngraph{u} \subseteq T \}.$ We define \emph{the $n'$--right regular bipartite subgraph induced by $T$} as the triple $$G_T  = (T, V_2(G)^T, E(G_T)),$$ where $E(G_T) = E(G) \cap (T \times V_2(G)^T).$
\end{defn}

\begin{lem}
Let $G$ be an $n'$-right regular bipartite graph. Let $T \subseteq V_1(G)$. Suppose $T$ is $k$--closed. For $S \subseteq T$ we have $$\cl{G}{k}{S} = \cl{G_T}{k}{S}.$$ \label{lem:subgraphforcing}
\end{lem}
\begin{proof}
The definition of $V_2(G)^T$ implies that $\ngraph{u} = \{v \in \vgraphone{G} \ | \ (v,u) \in \edgesgraph{G}  \}$ is equal to  $\ngraphT{u} = \{v \in \vgraphone{G_T} \ | \ (v,u) \in \edgesgraph{G_T}  \}$ for  $u \in V_2(G)^T$. As $S \subseteq T$ and $T$ is $k$--closed, Lemma \ref{lem:closedsubsets} implies $\cl{G}{k}{S} \subseteq T$. If there exists $u \in V_2(G)$ such that $\ngraph{u} \subseteq \cl{G}{k}{S}$, then $u \in V_2(G)^T$. Hence, $\ngraph{u} \subseteq \cl{G_T}{k}{S}$. This implies $\cl{G}{k}{S} \subseteq \cl{G_T}{k}{S}$. The set $\ngraphT{u}$ is also a neighborhood in $G$. The containment $\cl{G_T}{k}{S} \subseteq \cl{G}{k}{S}$ is clear. Thus, $$\cl{G}{k}{S} = \cl{G_T}{k}{S}.$$\end{proof}

\section{Encoding Tanner codes and irreversible $k$--threshold processes} 

We consider encoding a Tanner code as an irreversible $k$--threshold process. In this case, we require component codes where the values on any $k$ positions will determine the codeword uniquely (if such a codeword exists). In this case, $k$ may be larger than the dimension of the component code. In this regard, MDS codes are optimal because if a codeword of an MDS code of dimension $k$ has zeroes in any $k$ positions, then this codeword must be zero and there exists a codeword in the MDS codes with any $k$ prescribed values on any set of $k$ positions.

\begin{defn} \cite[Section 1]{ACDP14}
We say $S$ is a $k$--forcing set if $\cl{G}{k}{S} = \vgraphone{G}$.

\end{defn}

This is our generalization of an irreversible conversion set (see \cite{DR09}) and a $k$--forcing set as in \cite{ACDP14}. We are interested in $k$--forcing sets and their relation to encoding Tanner codes.

\begin{lem}Suppose $G$ is an $n'$-right regular bipartite graph and $C'$ is an MDS code of length $n'$ and dimension $k$. For $c \in (G,C')$, the complement of $supp(c)$ is $k$--closed.
\label{lem:zeroeskthreshold}\end{lem}
\begin{proof} Let $Z = supp(c)$ where $c \in (G,C')$. Suppose $u \in \vgraphtwo{G}$ satisfies $\#(Z \cap \ngraph{u}) \geq k.$ The vector $c^{\ngraph{u}}$ has at least $k$ zero positions. The projection $c^{\ngraph{u}}$ is a codeword of $C'_u$. As $C'_u$ is an MDS code of dimension $k$, $c^{\ngraph{u}}$  is the zero codeword, which implies $\ngraph{u} \subseteq Z.$ Therefore, there exists no $u \in V_2(G)$ such that $k \leq \#(Z \cap \ngraph{u}) < n'.$  Lemma \ref{lem:closedsetintersectionnum} implies $Z$ is $k$--closed. \end{proof}

\begin{thm}
Let $G$ be an $n'$-right regular bipartite graph. Let $C'$ be an MDS code of length $n'$ and dimension $k$. Then, the projection of $(G,C')^{\cl{G}{k}{S}}$ onto $S$ is a linear isomorphism between $(G,C')^{\cl{G}{k}{S}}$ and $(G,C')^{S}.$ 
\label{thm:forcingsetdimension}
\end{thm}

We interpret Theorem \ref{thm:forcingsetdimension} as an irreversible, $k$--threshold process. In this theorem, a codeword has only two types of positions: undetermined positions and positions determined to be zero. We consider an undetermined position to be in state $0$ and a zero position is in state $1$.

\begin{proof} Suppose $c \in (G,C')^{\cl{G}{k}{S}}$ is in the kernel of the projection of $(G,C')^{\cl{G}{k}{S}}$ onto $S$. Therefore, $c_s  = 0 $ for $s \in S$. Lemma \ref{lem:zeroeskthreshold} implies $c_s = 0 $ for $s \in \cl{G}{k}{S}$. Therefore, $c$ is the zero codeword of $(G,C')^{\cl{G}{k}{S}}$. Thus, the kernel of the projection $(G,C')^{\cl{G}{k}{S}}$ onto $(G,C')^S$ is trivial and $(G,C')^{S}$ is linearly isomorphic to $(G,C')^{\cl{G}{k}{S}}.$\end{proof}

Theorem \ref{thm:forcingsetdimension} states that if the positions in $S$ of $c \in (G, C')$ are determined to be zero, then the positions in $\cl{G}{k}{S}$ are also zero. Later,  we shall extend this to an encoding function. First, we obtain an upper bound on $\dim (G, C')$.

\begin{cor}

Let $G$ be an $n'$-right regular bipartite graph. Let $C'$ be an MDS code of length $n'$ and dimension $k$ and $S$ be a $k$-forcing set. Then, $\dim (G,C') \leq \# S$.
\label{cor:forcingsetdimension2}
\end{cor}
\begin{proof}
Theorem \ref{thm:forcingsetdimension} implies $$\dim (G,C')^{S} = \dim (G,C')^{\cl{G}{k}{S}}.$$ As $S$ is a $k$--forcing set,  we have $$\dim (G,C') = \dim (G,C')^S$$ which implies $\dim (G,C') \leq\# S.$\end{proof}

\begin{cor}

Let $G$ be an $n'$-right regular bipartite graph. Let $C'$ be an MDS code of length $n'$ and dimension $k$ and $S$ be a $k$-forcing set. Then, $S$ contains an information set for $(G,C)$.
\end{cor}
\begin{proof} Corollary \ref{cor:forcingsetdimension2} implies the projection onto $S$ is a code isomorphism. Hence, $S$ contains an information set for $(G, C')$. \end{proof}

If projecting a code $C \subseteq \ff{q}^A$ onto $B$ is a linear isomorphism, then the inverse map from $C^B$ to $C$ is a \emph{lengthening of $C^B$}. Now we define an iterative lengthening algorithm.
\begin{defn}

Let $G$ be an $n'$-right regular bipartite graph. Let $C'$ be an MDS code of length $n'$ and dimension $k$. Suppose $S$ be a $k$-forcing set. An \emph{iterative encoder} is a map which maps a vector $c^S \in (G,C')^S$ to its lengthened codeword $c \in (G,C').$ That is the inverse of the projection map of $(G,C')$ onto $S$.
\end{defn}

\begin{defn} For $C' \leq \ff{q}^{\mathcal{N}'}$ and vector $m \in \ff{q}^B$ where $B \subseteq \mathcal{N}'$, and $u \in \mathcal{N}'$ and $B$ contains an information set for $C'$  we define $PBIT(m, C', B, u)$ as the algorithm which returns the value at position $u$ of the unique codeword $c \in C'$ whose projection on $B$ is $m$. If such a codeword does not exist, then the algorithm fails.

For $x \in \ff{q}^{A}$ and $y \in \ff{q}^B$ we define the function $EXT(x,A, y, B)$ as the unique vector, if it exists $z \in \ff{q}^{A \cup B}$ whose projections satisfy $z^A =x$ and $z^B = y$.

\end{defn}

\begin{algorithm}
\begin{algorithmic}[5]
\STATE Input: $S \subseteq V_1(G)$, $m \in (G,C')^S$. 
\STATE $m' := m$.
\STATE $Z := S$  
\IF{$\exists u \in V_2(G) \ : \ \dim(C') \leq \# (\ngraph{u} \cap Z)  < n'$}
\STATE Set $T := \ngraph{u}$
\FOR{ $v \in T \setminus (T \cap Z)$ }
\STATE $m' := EXT(m', Z,  PBIT(m'^{T \cap Z}, C'_{u}, T \cap Z, v), v)$
\STATE $Z := Z \cup \{v\}$
\ENDFOR
\ENDIF
\STATE Return $m'$.

\end{algorithmic}
\caption{ $ENCODE(S, m)$: an iterative lengthening algorithm of $(G, C')^S$ to $(G,C)^{\cl{G}{k}{S}}$}
\label{alg:2}
\end{algorithm}

\begin{thm}
Let $G$ be an $n'$--right regular graph and $C'$ an MDS code of length $n'$. Suppose $S \subseteq V_1(G)$ and $m \in (G,C')^S$ are given as inputs to Algorithm \ref{alg:2}. Then Algorithm \ref{alg:2} outputs the unique codeword $c \in (G,C')^{\cl{G}{k}{S}}$ such that $c^S = m$.
\end{thm}
\begin{proof}
Suppose $c \in (G,C')^{\cl{G}{k}{S}}$.  We shall prove that if $c^S$ is given as input to  Algorithm \ref{alg:2} then $c$ will be the output. The vector $m'$ represents the codeword to be encoded and $S$ represents the known positions of the codeword.  We initialize $m'$ to the value $m$ and $Z$ is initialized to $S$.

We proceed by induction on $\#(\cl{G}{k}{S} \setminus S).$ Clearly, if $\cl{G}{k}{S} = S$, then the IF condition on line 4 is never satisfied, and the algorithm outputs $m'$ which was initially set to $c$

Now, suppose $\cl{G}{k}{S} \neq S$. Initially $m'$ equals the projection $c^S$. By the induction hypothesis we suppose the statement is true for $S'$  where $\#(\cl{G}{k}{S'} \setminus S') < \#(\cl{G}{k}{S} \setminus S).$ In this case, there exists some $u \in \vgraphtwo{G}$ such that $k \leq \# (\ngraph{u} \cap Z)  < n'.$ For each position $v \in T \setminus (T \cap Z)$, the algorithm determines the parity bit on that position given the values on $T \cap Z$. The codeword $m'$ is extended by this parity check bit on $v$. As $C'$ is an MDS code, each parity check bit is uniquely determined. Corollary \ref{cor:forcingsetdimension2} implies $m'$ is now the projection $c^{S \cup T}$. Likewise, $Z$ is updated to $S \cup T$. 

The next iteration of the algorithm is equivalent to invoking the algorithm with $c^{S \cup T}$ and  $S \cup T$. As $S \subsetneq S \cup T$ and $\cl{G}{k}{S} = \cl{G}{k}{S \cup T}$ the output is $c$ and $\cl{G}{k}{S}$ \end{proof}

We have reduced encoding $(G,C')$ to finding codewords in $(G,C')^S$. If $S \subseteq V_1(G)$ satisfies $\cl{G}{k}{S} = V_1(G)$ and $S$ is an information set for $(G,C')$ ,then we have an iterative and systematic encoder for $(G,C')$. This will be important for Schubert union codes.

\section{Schubert unions and related linear codes}

Here we present the Grassmannian, Schubert varieties and Schubert unions. All statements and definitions in this section are previously known. There are several references on the Grassmannian and Schubert varieties. We mention \cite{KL72} and \cite{S11}. First, we present the algebraic and geometric subspaces of the Grassmannian as in \cite{S11}. The relation among geometric and algebraic spaces was studied in \cite{S90}. Then we introduce the linear codes associated to the Grassmannian, Schubert varieties and Schubert unions. These linear codes are a standard technique to study the algebraic and geometric subspaces of any projective system with coding theory. For more on coding theory and projective systems one may consult \cite{NTV07}. For the remainder of this article, we fix $\ell$ and $m$ integers with $1 \leq \ell \leq m$.

\begin{defn}
 \emph{The Grassmannian}, $\grassmann{\ell}{m}$, is the set of all $\ell$ dimensional $\ff{q}$-linear subspaces of $\ff{q}^m$.
\label{defn:Grassmanndefn}
\end{defn}

Definition \ref{defn:Grassmanndefn} is an overly simplified definition of the Grassmannian. The classical way to study the Grassmannian is to focus on the Pl\"ucker relations over the algebraic closure of $\ff{q}$. However, we are working on linear codes of finite length. Thus, it is more useful to us to define the Grassmannian manifold as the solutions over $\ff{q}$ to the Pl\"ucker relations. Hence, we work only with the $\ff{q}$--rational points instead. This definition is equivalent.  We shall focus on the incidence geometry between the Grassmannian over $\ff{q}$ and its lines.

\begin{defn}

We define $$I(\ell, m) := \{  \alpha = (\alpha_1 < \alpha_2 < \ldots < \alpha_\ell) \in \mathbb{Z}^\ell \}$$

where $\{\alpha_1, \alpha_2, \ldots, \alpha_\ell\} \subseteq \{1,2,\dots, m\}$.
\end{defn}

The $\ell$--tuple $\alpha = (\alpha_1 < \alpha_2 < \cdots <\alpha_\ell)$ represents the set $\{ \alpha_i \ | \ 1 \leq i \leq \ell\}.$ We consider $I(\ell, m)$ as the set of all subsets of $\{1,2,\ldots, m\}$ of size $\ell$. 

\begin{defn}

Let $m$ be an integer. Suppose $\ell \leq m$. For $W \in \grassmann{\ell}{m}$ pick an $\ell \times m$ matrix whose rowspace is $W$. Denote this matrix by $\vectorname{M}_W$.  The map
\begin{center}
\begin{tabular}{rrcl}
$ev:$ &  $\grassmann{\ell}{m}$ & $\rightarrow$ &  $\projectivespace{\binom{m}{\ell}-1}{\ff{q}}$ \\  $ev:$ & $W$ & $\mapsto$  & $(\mathrm{det}_{I}(\vectorname{M}_{W}))_{I \in I(\ell, m)}$\\
\end{tabular}
\end{center}  is known as the \emph{Pl\"ucker embedding}.

\end{defn}

The Pl\"ucker embedding is a nondegenerate embedding of $\grassmann{\ell}{m}$ into $\projectivespace{\binom{m}{\ell}-1}{\ff{q}}$. The image of $\grassmann{\ell}{m}$ is known as the \emph{Grassmann variety}. The Grassmann variety has highly desirable algebraic and geometric properties. 

\begin{defn}\cite{S90} Let $A$ be a linear subspace of $\ff{q}^{I(\ell, m)}$ of dimension $r$. An \emph{algebraic subspace of $\grassmann{\ell}{m}$} is a subset of $\grassmann{\ell}{m}$ of the form $$\{ W \in \grassmann{\ell}{m} \ | \ \sum_{I \in I(\ell, m)} a_I \mathrm{det}_I(\vectorname{M}_{W}) = 0 \ \forall \ a \in A \}$$ or equivalently $\{ W \in \grassmann{\ell}{m} \ | \ ev(W) \in A^\perp \}$ where $A^\perp$ is the orthogonal complement of $A$. 

If $A$ has dimension $1$, then the algebraic subspace defined by $A$ is an \emph{algebraic hyperplane of $\grassmann{\ell}{m}$}. \end{defn}

\begin{defn}\cite[Chapter 3, Example 6]{S11} A \emph{line of $\grassmann{\ell}{m}$} is a set of the form $$\pi_{Z}^{Z'} := \{  V \in \grassmann{\ell}{m} \ | \ Z \subseteq V \subseteq Z'\},$$ where $Z \in \grassmann{\ell-1}{m}$, $Z' \in \grassmann{\ell+1}{m}$ and $Z \subseteq Z'$. We denote by $\lines(\grassmann{\ell}{m})$ the set of all lines of the Grassmannian.\end{defn}

The lines $ev(\pi_Z^{Z'}) \subseteq \projectivespace{\binom{m}{\ell}-1}{\ff{q}}$ are the lines of $\projectivespace{\binom{m}{\ell}-1}{\ff{q}}$ contained in $ev(\grassmann{\ell}{m})$. The geometry of the Grassmannian can be defined from the incidence geometry from the elements of the Grassmannian and the lines of the Grassmannian.

\begin{defn}\cite[Section 3.4]{S11} A \emph{geometric subspace of the Grassmannian} is a subset $X \subseteq \grassmann{\ell}{m}$ such that any line of the Grassmannian  has either $0$, $1$ or $q+1$ points in $X$.

\cite[Section 4.1.2]{S11} A \emph{geometric hyperplane of the Grassmannian} is a subset $X \subseteq \grassmann{\ell}{m}$ such that any line of the Grassmannian either has either $1$ or $q+1$ points in $X$.
\label{defn:geometricsubspaces}
\end{defn}
What we call a \emph{geometric subspace}, E. Shult, in \cite{S11} has called a \emph{subspace}. Shult's original statement of Definition \ref{defn:geometricsubspaces} also applies to infinite fields. The key result of \cite{S90} is that algebraic and geometric subspaces are equivalent.

\begin{prop}\cite[Theorem 2]{S90} Let $X \subseteq \grassmann{\ell}{m}$. The set $X$ is an algebraic subspace of the Grassmannian if and only if $X$ is a geometric subspace of the Grassmannian. The set $X$ is an algebraic hyperplane of the Grassmannian if and only if $X$ is a geometric hyperplane of the Grassmannian.
\label{prof:geomsubspace} \end{prop}

We state two properties of geometric spaces and lines which we will need for Schubert varieties and Schubert unions.

\begin{defn}\cite[Section 3.4.1]{S11} Let $X$ be a geometric subspace of $\grassmann{\ell}{m}$, A \emph{line of $X$} is a line of $\grassmann{\ell}{m}$ which is contained in $X$. We denote the set of all such lines by $\lines(X)$.
\end{defn}

\begin{prop}\cite[Chapter 3]{S11} The intersection of two geometric subspaces is a geometric subspace. For any two subsets $A, B \subseteq \grassmann{\ell}{m}$, there exists an unique smallest geometric subspace containing $A \cup B$. 

\end{prop}

Some interesting subgraphs arise from the incidence between points of a geometric subspace $X$ and its lines $\lines(X)$. We shall study some Tanner codes made from these graphs.

\begin{defn}\cite[Section 3.4.1]{S11} We define the point--line incidence graph of $\grassmann{\ell}{m}$ as the bipartite graph $\Gamma = (\grassmann{\ell}{m}, \lines(\grassmann{\ell}{m}), E(\Gamma))$ where the edge set $\edgesgraph{\Gamma} := \{ (W, L) \in \grassmann{\ell}{m} \times \lines(\grassmann{\ell}{m})  \ | \  W \in L  \}.$ 

Let $X$ be a geometric subspace of $\grassmann{\ell}{m}$. We define the point--line incidence subgraph of $X$ as $\Gamma_X =  (X, \lines(X), E(\Gamma_X))$ where the edge set $\edgesgraph{\Gamma_X} := \{ (W,L) \in X \times \lines(X)  \ | \  W \in L  \}.$ \label{defn:geometrygraphsubspace}\end{defn}

The original definition of a Schubert variety is  given by $\{ W \in \grassmann{\ell}{m} \ | \  \dim (W \cap A_i) \geq i \}$, where $A_{1} \subsetneq A_2 \subsetneq \cdots \subsetneq A_\ell \subseteq \ff{q}^m$. In \cite{KL72} Schubert varieties were determined to be algebraic subspaces. We use this algebraic description. The following partial order on $I(\ell, m)$ is needed to define Schubert varieties.

\begin{defn}
Let $\alpha, \beta \in I(\ell,m)$. The \emph{Bruhat order} is the partial order on $I(\ell, m)$ defined by

$$\alpha \leq \beta \makebox{ if and only if } \alpha_i \leq \beta_i \makebox{ for } i=1,2,\ldots, \ell.$$

\label{def:SchubertOrder}

A subset $S \subseteq I(\ell, m)$ is said to be \emph{downward closed} if $\alpha \leq \beta$ and $\beta \in S$ imply $\alpha \in S$.
\end{defn}

\begin{defn}\cite{KL72} Let $\alpha \in I(\ell,m)$. The \emph{Schubert variety corresponding to $\alpha$} is $$\Omega_\alpha := \{ W \in \grassmann{\ell}{m} \ | \  \mathrm{det}_{I}(\vectorname{M}_{W}) = 0, \forall I \not\leq \alpha \}.$$\end{defn}

Note that $\Omega_\alpha$ is the algebraic subspace defined by $A_\alpha = span( \{ e_I  \ | I \not\leq \alpha \}) \leq \ff{q}^{I(\ell, m)}.$ 

\begin{defn}\cite[Proposition 2.6]{HJR07} Let $S$ be a downward closed subset of $I(\ell,m)$. A \emph{Schubert union} is a subset of the form $$\Omega_S := \bigcup\limits_{\alpha \in S} \Omega_\alpha.$$

\end{defn}

The linear subspace $$A_S := span( \{ e_{I}   \ | I \not\leq \alpha, \forall \alpha \in S \}) \leq \ff{q}^{I(\ell, m)}$$ also determines $\Omega_S.$ We denote by $$\{ e_\alpha | \alpha \in I(\ell, m)\}$$ the standard basis of $\ff{q}^{I(\ell, m)}$. 

\begin{defn}

Suppose $\alpha  \in I(\ell,m)$ let $W_\alpha$ be the row space of the $\ell\times m$ matrix whose $(i,\alpha_i)$--th entry is equal to $1$ and all other entries are $0$.  Let $S \subseteq I(\ell, m)$ be a downward closed subset. We define $$J_S := \{ W_\beta \  | \  \beta \in S\}.$$ 

\end{defn}

Note that $ev(W_\alpha) \in span({e_{\alpha}}) \leq \ff{q}^{I(\ell, m)}.$ The set $J_{I(\ell, m)}$ is known as an apartment of the Grassmannian, see \cite{P10}.  The sets $J_S$ will be important for encoding Schubert union codes.

\begin{defn}\cite{R87a} For each $W \in \grassmann{\ell}{m}$ pick an $\ell \times m$ matrix whose rowspace is $W$. Denote this matrix by $\vectorname{M}_W$.  Define $G$ as the following $\# I(\ell,m) \times \#\grassmann{\ell}{m}$ matrix:

$$ G :=  \left(\mathrm{det}_I(\vectorname{M}_W) \right)_{I \in I(\ell, m), W \in \grassmann{\ell}{m}}.$$

The Grassmann code $\grassmanncode{\ell}{m}$ is defined as the rowspace of $G$. \label{def:GrassmannEvaluationCode}
\end{defn}

The Grassmann code is a code from the projective system of the Pl\"ucker embedding, which maps $\grassmann{\ell}{m}$ into $\projectivespace{\binom{m}{\ell}-1}{\ff{q}}$. Grassmann codes were introduced in \cite{R87a, R87b} and \cite{N93}. The parameters of the Grassmann code are $[\# \grassmann{\ell}{m}, \binom{m}{\ell}, q^{\ell(m-\ell)}]$. As the elements $W \in J_{I(\ell, m)}$ are mapped to the standard basis vectors,  $J_{I(\ell, m)}$ is an information set for  $\grassmanncode{\ell}{m}$.
\begin{defn}\cite{GL00} A \emph{Schubert code} is the projection of $\grassmanncode{\ell}{m}$ onto $\Omega_\alpha$, that is:

$$\schubertcodealpha{\ell}{m}{\alpha} := \grassmanncode{\ell}{m}^{\Omega_\alpha}.$$

\end{defn}
Schubert codes were introduced in \cite{GL00}. The dimension of Schubert codes was found to be $\# J_\alpha$. In \cite[Theorem 2]{X08}, Xiang determined that the minimum distance is $q^{\delta_\alpha}$, where $\delta_\alpha := \sum (\alpha_i - i)$.

\begin{defn}\cite{HJR07} Let $S \subseteq I(\ell,m)$.  A \emph{Schubert union code} is defined as the projection of $\grassmanncode{\ell}{m}$ on $\Omega_S$. That is $$\schubertcodealpha{\ell}{m}{S} := \grassmanncode{\ell}{m}^{\Omega_S}.$$ \end{defn}

Schubert union codes were introduced in \cite{HJR07}. There the dimension of $\schubertcodealpha{\ell}{m}{S}$ was found to be $\# J_S$. The same authors found the minimum distance of Schubert union codes for $\ell=2$ in \cite{HJR09}. In the next section, we find the true minimum distance of Schubert union codes. 

\begin{prop}
Let $S$ be a downward closed subset of $I(\ell,m)$. The set $J_S$ is an information set for $\schubertcodealpha{\ell}{m}{S}.$

\end{prop}
\begin{proof} Note that $J_S = J_{I(\ell, m)}\cap \Omega_S.$ Therefore, $\schubertcodealpha{\ell}{m}{S}^{J_S} = \schubertcodealpha{\ell}{m}{S}^{J_{I(\ell, m)} \cap \Omega_S} = \ff{q}^{J_S}.$ As $\dim\schubertcodealpha{\ell}{m}{S}^{J_S} = \# J_S$ the claim follows. \end{proof}

\section{Schubert union codes as Tanner codes}

Now we shall determine the structure of Schubert union codes and their duals.  Our main result is that the Schubert union code is a Tanner code from their point--line incidence graph as given in Definition \ref{defn:geometrygraphsubspace}. The sets $J_S$ are very important. 

\begin{prop}
The Pl\"ucker embedding maps a line $L \in \lines(\grassmann{\ell}{m})$ to a line in $\projectivespace{\binom{m}{\ell}}{\ff{q}}$.
\end{prop}
\begin{proof} Consider $U, W \in L$, where $L \in \lines(\grassmann{\ell}{m}).$ As $\dim (U \cap W) = \ell-1$, $U$ and $W$ have representative matrices of the form $\vectorname{M}_{U}$ and $\vectorname{M}_{W}$, where these matrices coincide on the first $\ell-1$. Let $u$ be the last row of $\vectorname{M}_{U}$ and $w$ be the last row of $\vectorname{M}_{W}$. Any $T \in L$, $L \neq W$ is of the form $span(U \cap W, u + \beta w)$, where $\beta \in \ff{q}$. There exists a representative matrix for $T$ of the form $\vectorname{M}_{T}$ where the first $\ell-1$ rows are the same rows as $\vectorname{M}_{U}$ and $\vectorname{M}_{W}$ and the last row is $u + \beta w$. The multilinearity of the determinant implies $\mathrm{det}_{I}(\vectorname{M}_{T}) = \mathrm{det}_{I}(\vectorname{M}_{U}) + \beta \mathrm{det}_{I}(\vectorname{M}_{W})$ for $I \in I(\ell, m)$. Hence, the vectors $ev(T)$ are in the line containing $ev(U)$ and $ev(W)$. In this way, $L$ is identified with $\projectivespace{1}{\ff{q}}$ as a line of  $\projectivespace{\binom{m}{\ell}-1}{\ff{q}}$.\end{proof}

 The linear code corresponding to $\projectivespace{1}{\ff{q}}$ is the doubly extended Reed--Solomon code of dimension $2$. Its parameters are $[q+1,2,q]$. It is an MDS code of dimension $2$. As the Pl\"ucker embedding  maps $\grassmann{1}{2}$ to $\projectivespace{1}{\ff{q}}$, the code $\grassmanncode{1}{2}$ is a doubly extended Reed--Solomon code  corresponding to $\projectivespace{1}{\ff{q}}$.

\begin{lem}
Let $S$ be a downward closed subset of $I(\ell,m)$. There exists a Tanner code $$C = (\Gamma_{\Omega_S}, \grassmanncode{1}{2})$$ such that, $$\schubertcodealpha{\ell}{m}{S}  = C. $$\label{lem:SchubertTannerSubcode}
\end{lem} \begin{proof} In view of Proposition \ref{prof:geomsubspace}, the point--line incidence graph $\Gamma_{\Omega_S}$ is a $(q+1)$--right regular graph. For any $L \in \lines(\Omega_S)$ the neighborhood $\ngraph{L}$ is simply the line $L$, which is identified with $\projectivespace{1}{\ff{q}}$. Therefore, the code $\schubertcodealpha{\ell}{m}{S}^L$ is isomorphic to $\grassmanncode{1}{2}$ and $\schubertcodealpha{\ell}{m}{S} $ is a Tanner code $(\Gamma_{\Omega_S}, \grassmanncode{1}{2})$. \end{proof}

We have the following lemma on $2$--closed sets of $\grassmann{\ell}{m}$ and geometric subspaces.

\begin{lem}

Let $X \subseteq \grassmann{\ell}{m}$. The set $X$ is $2$--closed with respect to $\Gamma$ if and only if $X$ is a geometric subspace.
\label{lem:2forcinggeom}
\end{lem}
\begin{proof} Lemma \ref{lem:closedsetintersectionnum} implies $X$ is $2$--closed if and only if
  $X$ intersects a line in $0$, $1$ or $q+1$ places. This is precisely the statement that $X$ is a geometric subspace. \end{proof}

\begin{lem}

Let $Y \subseteq X \subseteq \grassmann{\ell}{m}$. The set $Y$ is $2$--closed with respect to $\Gamma_X$ if and only if $Y$ is a geometric subspace of the Grassmannian.
\label{lem:2forcinggeom2}
\end{lem}
\begin{proof} Derived from Lemmas \ref{lem:subgraphforcing} and \ref{lem:2forcinggeom}. \end{proof}

\begin{lem}

Let $X \subseteq  \grassmann{\ell}{m}$. The smallest geometric subspace containing $X$ is $ \cl{\Gamma}{2}{X}$.
\end{lem}
\begin{proof} Clearly, Lemma \ref{lem:2forcinggeom} implies $ \cl{\Gamma}{2}{X}$ is a geometric subspace containing $X$. Now suppose $T$ is a geometric subspace containing $X$. Then, $\cl{\Gamma}{2}{X} \subseteq \cl{\Gamma}{2}{T}$. However, as $T$ is a geometric subspace, we have $\cl{\Gamma}{2}{T} = T$, which implies $ \cl{\Gamma}{2}{X}$ is contained in all geometric spaces containing $S$. \end{proof}

In theory, $\cl{\Gamma}{2}{S}$ may be found with Algorithm \ref{alg:1}. As, $\Gamma$ has ${ m \brack \ell+1 }_q{\ell+1 \brack 2}_q$ lines, at worst one would need to compare around  $({ m \brack \ell+1 }_q{\ell+1 \brack 2}_q)^2$ lines.  However $\cl{\Gamma}{2}{S}$ is the smallest geometric subspace containing $S$. One could find  $\cl{\Gamma}{2}{S}$ as a geometric or algebraic subspace instead.  In this case, one would need to check the basis of $span(ev(S))^\perp$ with all ${ m \brack \ell}_q$ elements of  $\grassmann{\ell}{m}$.

\begin{thm}Let $S$ be a downward closed subset of $I(\ell,m)$. Then, $$\cl{\Gamma_{\Omega_S}}{2}{J_S} = \Omega_S $$\label{thm:closureunion}\end{thm}

\begin{proof}Note that  $J_S \subseteq \Omega_S$. As Lemma \ref{lem:2forcinggeom} implies $\cl{\Gamma_{\Omega_S}}{2}{J_S}$ is a geometric subspace and Proposition \ref{prof:geomsubspace} implies $\Omega_S$ is a geometric subspace we conclude that $\cl{\Gamma_{\Omega_S}}{2}{J_S} \subseteq \Omega_S.$ 

As $\Omega_S$ is the algebraic subspace defined by $\mathrm{det}_I(\vectorname{M}_W) = 0$ for $ I \not\leq \alpha,$ for all $\alpha \in S$ it follows that $$ev(\Omega_\alpha) =  ev(\grassmann{\ell}{m}) \cap span(\{ e_{\beta} \ | \ \beta \in S \} ).$$ As $\cl{\Gamma_{\Omega_S}}{2}{J_S}$ is a geometric subspace,   Proposition \ref{prof:geomsubspace} implies $\cl{\Gamma_{\Omega_S}}{2}{J_S}$ is an algebraic subspace. As $\cl{\Gamma_{\Omega_S}}{2}{J_S}$ is a subspace of $\Omega_S$,  the set $ev(\cl{\Gamma_{\Omega_S}}{2}{J_S})$ is contained inside the set $\{x \in \ff{q}^{I(\ell, m)} \ | \ x_{\alpha} = 0, \alpha \not\in S\}$. Recall that the vector  $ev(W_\beta) = e_\beta \in \ff{q}^{I(\ell, m)}$. Hence, $ev(\cl{\Gamma_{\Omega_S}}{2}{J_S}) \subseteq \ff{q}^{I(\ell, m)}$ has a basis for  vector space of dimension $\# S$. Thus,$\cl{\Gamma_{\Omega_S}}{2}{J_S}$ satisfies no additional linear equations which implies $\cl{\Gamma_{\Omega_S}}{2}{J_S} = \Omega_S$.     \end{proof}

As $J_S$ is a $2$--forcing set for $\Gamma_S$ we obtain the following.

\begin{thm}\label{thm:SchubertTannercode}
Let $S$ be a downward closed subset of $I(\ell,m)$. If $C$ is a Tanner code of the form $C = (\Gamma_{\Omega_S}, \grassmanncode{1}{2})$ such that $\schubertcodealpha{\ell}{m}{S} \subseteq C,$ then 
$$\schubertcodealpha{\ell}{m}{S} =C.$$
\end{thm}
\begin{proof} Let $C =  (\Gamma_{\Omega_S}, \grassmanncode{1}{2})$ such that $\schubertcodealpha{\ell}{m}{S} \subseteq C$. Then, $\# J_S = \dim \schubertcodealpha{\ell}{m}{S}.$ As $J_S$ is a $2$--forcing set for $\Gamma_{\Omega_S}$, Corollary \ref{cor:forcingsetdimension2} implies $\dim C  \leq \# J_S$.  Therefore, $$\schubertcodealpha{\ell}{m}{S} = C.$$  \end{proof}

\begin{cor}
Algorithm \ref{alg:2}  extends a message $m \in \ff{q}^{J_S}$ a codeword $c \in \schubertcodealpha{\ell}{m}{S}$ with an encoder for $\grassmanncode{1}{2}$ and the lines used to find $\cl{\Gamma_{\Omega_S}}{2}{J_S}$.
\end{cor}

We can use this encoder to find the minimum distance of $\schubertcodealpha{\ell}{m}{S}$. Recall that, for $\alpha \in I(\ell, m)$,  $\delta_\alpha$ denotes the sum $\sum (\alpha_i - i)$ and that $q^{\delta_\alpha}$ is the minimum distance of $\schubertcodealpha{\ell}{m}{\alpha}$.

\begin{thm}
Let $S$ be a downward closed subset of $I(\ell,m)$. Suppose $S'$ is the set of the maximal elements of $S$. Then, the minimum distance of $\schubertcodealpha{\ell}{m}{S}$ is $min_{\alpha \in S'}\{ q^{\delta_\alpha} \}$.
\end{thm}

\begin{proof}

Let $c$ be a nonzero codeword of $\schubertcodealpha{\ell}{m}{S}$. As $J_S$ is an information set of $\schubertcodealpha{\ell}{m}{S}$, there exists $W \in J_S \cap supp(c).$ This $W \in J_\alpha$ for some $\alpha \in S'$. Therefore, the projection $c^{J_{\alpha}}$ is encoded to a nonzero codeword of $\schubertcodealpha{\ell}{m}{\alpha}$. By a theorem of Xiang \cite[Theorem 2]{X08} the weight of $c^{\Omega_\alpha}$ is at least $q^{\delta_\alpha}$. Therefore, the minimum distance of $\schubertcodealpha{\ell}{m}{S}$ is at least $min_{\alpha \in S'}\{ q^{\delta_\alpha} \}$. 

We prove equality by finding a codeword whose weight is exactly $min_{\alpha \in S'}\{ q^{\delta_\alpha} \}$. Consider the codeword $c \in \schubertcodealpha{\ell}{m}{S}$ corresponding to $\mathrm{det}_\alpha$ for $\alpha \in S'$. As  $\mathrm{det}_\alpha(W_\beta) = 0$ for $\beta \in  S' \setminus \{ \alpha\}$, Theorem \ref{thm:SchubertTannercode} implies $c^{\Omega_\beta}$ is the zero codeword. Therefore, $supp(c) \subseteq \Omega_\alpha$. As $c^{\Omega_\alpha}$ has $q^{\delta_\alpha}$ nonzero entries, $c^{\Omega_\alpha}$ is a minimum weight codeword of $\schubertcodealpha{\ell}{m}{\alpha}$. Hence, $\# supp(c) = q^{\delta_\alpha}$.  \end{proof}

In  \cite[Theorem 30]{BP15}  it was established that $\grassmanncode{\ell}{m}^\perp$ is generated by its minimum weight codewords. That proof is based on \cite[Theorem 34]{BGH12}, which proves that the dual codes of affine Grassmann codes are generated by their minimum weight codewords. We establish that  $\schubertcodealpha{\ell}{m}{S}^\perp$ is generated by its minimum weight codewords from the fact that Schubert union codes are maximal Tanner codes.

\begin{cor}
Let $S$ be a downward closed subset of $I(\ell,m)$. The code $\schubertcodealpha{\ell}{m}{S}^\perp$ is generated by its minimum weight codewords.
\end{cor}
\begin{proof} Lemma \ref{lem:SchubertTannerSubcode} and Theorem \ref{thm:SchubertTannercode} imply $\schubertcodealpha{\ell}{m}{S}$ is a Tanner code with component code $ \grassmanncode{1}{2}$ and associated bipartite graph $\Gamma_{\Omega_S}$. Lemma \ref{lem:MaxTannerthm}  implies $\schubertcodealpha{\ell}{m}{S} \subseteq (G,C')_{\phi, \mathcal{C}}$ for the $\# V_2(G)$--tuples $\phi = (\phi_u) $ and $\mathcal{C}= (C'_u)$. Theorem \ref{thm:SchubertTannercode} implies $\schubertcodealpha{\ell}{m}{S} = (G,C')_{\phi, \mathcal{C}}$.

The code $(G,C')_{\phi, \mathcal{C}}^\perp$ is generated by $span(\{ D_u \ | \ u \in V_2(\Gamma_{\Omega_S})\} )$ where $D_u$ is the code in $\ff{q}^{V_1(\Gamma_{\Omega_S})}$ whose projection on $\ngraph{u}$ is $(\schubertcodealpha{\ell}{m}{S}^{\ngraph{u}})^\perp = (C'_u)^\perp$ and $supp(d) \subseteq \ngraph{u}$ for $d \in D_u$. Recall that the codewords of $D_u \leq \ff{q}^{\Omega_S}$ are zero outside of $\ngraph{u}$, and $D_u^{\ngraph{u}}  = \grassmanncode{1}{2}^\perp$. As $\grassmanncode{1}{2}^\perp$ is a $[q+1, q-1, 3]$ code which is generated by its minimum weight codewords, each $D_u$ is spanned by its weight $3$ codewords, $\schubertcodealpha{\ell}{m}{S}^\perp$ is also generated by its minimum weight codewords. \end{proof}

\begin{cor}
Let $S$ be a downward closed subset of $I(\ell,m)$. Suppose $\{ (1,2,\ldots, \ell)\} \neq S$. The minimum weight of  $\schubertcodealpha{\ell}{m}{S}^\perp$ is  $3$.
\end{cor}
\begin{proof}

It follows easily from the fact that $\{ (1,2,\ldots, \ell)\} \neq S$ implies $\Omega_S$ contains a line. Therefore, there is a weight $3$ codeword of $\grassmanncode{\ell}{m}^\perp$ whose support is contained in $\Omega_S$. As $\schubertcodealpha{\ell}{m}{S}^\perp$ is a shortening of  $\grassmanncode{\ell}{m}^\perp$ , the statement follows easily. \end{proof}

\section{Graph bounds applied to Schubert unions and related codes}

We have expressed Schubert union codes as Tanner codes with the point--line incidence structure Schubert unions inherit from the Grassmannian as their bipartite graph. This gives an iterative encoding algorithm and a proof that the dual codes of Schubert union codes are generated by their minimum weight codewords. However, Tanner codes are desirable for their iterative decoding algorithm and minimum distance bounds.  These bounds are derived from graph expansion. In this section, we use the Grassmann graph and its eigenvalues to find expressions for the minimum distance of the Grassmann code. In fact, in this way we attain both a lower bound and an upper bound for the weight of a codeword of  $\grassmanncode{\ell}{m}$. Our technique is different than the classical technique used to determine the minimum distance. We start by finding divisibility conditions on the weights of the codewords of Grassmann codes. These conditions improve the bounds from the eigenvalues of the Grassmann graph.

\begin{prop}\cite{S90} Let $Y$ be the support of a codeword, $c$, of $\grassmanncode{\ell}{m}$. Any line of $\grassmanncode{\ell}{m}$ has either $0$ or $q$ points in common with $Y$.\end{prop}
\begin{proof}Consider the projection $c^L$ for any line $L \in \lines(\grassmann{\ell}{m})$. As $c^{L}$ is a codeword of a $[q+1, 2, q]$ code, it is either the zero codeword, or it has weight $q$.\end{proof}

\begin{defn}\cite{P10, GPP09} For $S \in \grassmann{\ell-1}{m}(V)$ and $T \in \grassmann{\ell+1}{m}(V)$ define $$ [S\rangle  := \{ W \in \grassmann{\ell}{m} \ | S \subseteq W \}$$ and  $$\langle T] := \{ W \in \grassmann{\ell}{m} \ | W \subseteq T \}.$$ The sets are known as \emph{star cliques} or \emph{top cliques} respectively.
\end{defn}

In \cite{GPP09}, the sets  $[S| \rangle$ and $\langle| T]$, where  the linear spaces $S \in \grassmann{\ell-1}{m}(V)$ and $T \in \grassmann{\ell+1}{m}(V)$ are known as \emph{close families of type I and type II} respectively. This is because any two $\ell$--spaces in either clique set are adjacent in the Grassmann graph. Any line $\pi_S^T \in \lines(\grassmann{\ell}{m})$ equals $[S |\rangle \cap \langle| T]$, where $S \subseteq T$. Any two $\ell$ spaces $W$ and $W'$ such that $\dim W \cap W' = \ell -1$ are contained in a unique line. The authors in \cite{GPP09} find lower bounds for some generalized Hamming weights of Grassmann codes. They find these bounds by determining the number of indecomposable vectors of $\ff{q}$--linear subspaces of the wedge product. 

  Note that the Pl\"ucker embedding identifies $\langle T]$ with $\projectivespace{\ell+1}{\ff{q}}$ and $[S \rangle$  with $\projectivespace{m-\ell+1}{\ff{q}}$. Recall that the linear code associated to  $\projectivespace{\ell+1}{\ff{q}}$ is the Projective Reed--Muller code of degree $1$. It is a $[{{\ell+1 \brack 1}_q},\ell+1, q^\ell]$ code. Its dual code is a $[ {{\ell+1 \brack 1}_q}, {{\ell+1 \brack 1}_q}  -  (\ell+1), 3]$ cyclic code. Both codes are generated by their minimum weight codewords. 

\begin{prop}\cite[Lemma 5]{GPP09}

 Let $Y$ be the support of a codeword, $c$, of $\grassmanncode{\ell}{m}$. Let  $T \in \grassmann{\ell+1}{m}$.  The top clique $\langle T]$ has either $0$ or $q^{\ell}$ points in common with $Y$. If $S \in \grassmann{\ell-1}{m}$, then $[S\rangle$ has either $0$ or $q^{m- \ell}$ points in common with $Y$.
\label{prop:ReedMullerIntersection}
\end{prop}\begin{proof} Lemma 5 of \cite{GPP09} implies a close family is mapped to some $\projectivespace{r}{\ff{q}}$. In this case, $r = \ell+1$ or $r = m-\ell+1$. The linear code associated to $\projectivespace{r}{\ff{q}}$ is a Reed--Muller code of degree $1$. The nonzero positions of any nonzero codeword of this Reed--Muller are the complement of a hyperplane of $\projectivespace{r}{\ff{q}}$. Therefore all nonzero codewords have weight $q^{r-1}$.\end{proof}

\begin{thm}

Let $Y = supp(c)$ where $c$ is a codeword of $\grassmanncode{\ell}{m}$. Both $q^{\ell}$ and  $q^{m -\ell}$ divide $\# Y$.
\label{thm:divisibilitybound}
\end{thm}
\begin{proof}

Proposition \ref{prop:ReedMullerIntersection}  implies $Y$ is a geometric space. Let $$a_1(Y) := \# \{ (Y_1, Y_2) \in Y^2 \ | \ \dim Y_1 \cap Y_2 = \ell -1\}.$$ For $W \in Y$ there are ${m - \ell \brack 1}_q{\ell \brack 1}_q$ lines which contain $W$. These lines have exactly $q-1$ points in $Y$ different from $W$. Therefore, $$a_1(Y) = \# Y (q-1){m - \ell \brack 1}_q{\ell \brack 1}_q.$$

Let $f_1$ be the number top cliques with an element in $Y$. Each of those cliques has $q^{\ell}$ points in $Y$. Therefore, $$a_1(Y) = f_1q^{\ell}(q^{\ell}-1).$$ Likewise, if $f_2$ is the number of star cliques  with an element in $Y$, then $$a_1(Y) = f_2q^{m-\ell}(q^{m-\ell}-1).$$  As $q$ does not divide $(q-1){m - \ell \brack 1}_q{\ell \brack 1}_q$, it follows that $q^\ell$ and $q^{m-\ell}$ divide $\# Y$. \end{proof}

As the Grassmann code $\grassmanncode{\ell}{m}$ is a Tanner code, we also use eigenvalue arguments to find bounds on the size of the support of codewords of $\grassmanncode{\ell}{m}.$ As the colineation graph of the point--line incidence graph of the Grassmannian, (i.e. the Grassmann graph), is a distance transitive graph, 
we have the following theorem due to Eisfield. This bound is similar to the minimum distance bound of a Tanner code $(G,C')$ using the eigenvalues of its adjacency matrix.

\begin{prop}\cite{E98} Let $2\ell \leq m$ and $Y \subseteq \grassmann{\ell}{m}$. Suppose $a_1(Y)$ is tge number of ordered pairs of close vectors in $Y$, thay is $$a_1(Y) := \# \{ (Y_1, Y_2) \in Y^2 \ | \dim Y_1 \cap Y_2 = \ell -1\}$$  then
$$\frac{\# Y}{{m \brack \ell}_q}\left( \left(\theta_0-\theta_\ell\right) \# Y + \theta_\ell{m \brack \ell}_q \right)  \leq  a_1(Y)$$ and $$a_1(Y) \leq \frac{\# Y}{{m \brack \ell}_q}\left( \left(\theta_0-\theta_1\right) \# Y +  \theta_1{m \brack \ell}_q \right)$$

Where $$\theta_0 = q{m - \ell \brack 1}_q{\ell \brack 1}_q,$$ $$\theta_1 =  q^2{m - \ell-1 \brack 1}_q{\ell-1 \brack 1}_q-1 $$ and  $$\theta_\ell = - {\ell \brack 1}_q.$$
\label{prop:Eisfeld}
\end{prop}

Proposition \ref{prop:Eisfeld} gives lower and upper bounds for $\# supp(c)$, where  $c \in \grassmanncode{\ell}{m}$.

\begin{thm}

Let $Y = supp(c)$ where $c$ is a nonzero codeword of $\grassmanncode{\ell}{m}$. Then,  $${m \brack \ell}_q - {m-\ell \brack 1}_q{m-1 \brack \ell-1}_q  \leq \# Y$$ and $$ Y   \leq \frac{ q^{m-\ell}}{{m - \ell +1\brack 1}_q}{m \brack \ell}_q.$$

\end{thm}

\begin{proof} Substitute $$\# a_1(Y) :=  \# Y (q-1){m - \ell \brack 1}_q{\ell \brack 1}_q$$ in Proposition \ref{prop:Eisfeld}  and solve for $\# Y$. \end{proof}


By expressing Grassmann codes as a Tanner code, Theorem \ref{thm:divisibilitybound} gives a lower bound on the minimum weight of the Grassmann code  from the eigenvalues of the Grassmann graph. For example, D. Nogin in \cite{N93} intersects Grassmannians with linear varieties. In \cite{GPP09}, the authors determined the least number of decomposable vectors in a linear subspace of dimension $r$ of the wedge product. This is a geometric subspace of the Grassmannian.  Principally, they give a lower bound for $ev(\grassmann{\ell}{m}) \cap D$ in terms of $\dim D$.  We expect Schubert unions and  higher weights of the Grassmann code could be studied with Proposition \ref{prop:Eisfeld}. Geometric subspaces will be important in this regard.  For example, we have the following corollary from \cite{GPP09, BP15} and  \cite{JL07}.

\begin{lem}

Let $d < d'$  be two successive Generalized Hamming weights of $\schubertcodealpha{\ell}{m}{S}^\perp$. Then, $1 \leq d' - d \leq 2$.

\end{lem}
\begin{proof}
Note that the minimum distance $\schubertcodealpha{\ell}{m}{S}^\perp$ is three. \cite{GPP09, BP15} imply that for each element in $\Omega_S$ there is a minimum weight codeword with that position in its support.  \cite{JL07} implies $d'-d \leq 2$.\end{proof}

As any successive generalized Hamming weights of  $\schubertcodealpha{\ell}{m}{S}^\perp$ increase by $1$ or $2$, we need to determine the numbers where the increase is $2$ instead of $1$. This skip gives one of the Generalized Hamming weighs of $\schubertcodealpha{\ell}{m}{S}$.  In fact, we suspect that the linear subspaces of $\schubertcodealpha{\ell}{m}{S}^\perp$ whose support is $2$--closed will play an important role.


\section*{Conclusion}

We have used irreversible $k$--threshold processes to simulate encoding a Tanner code. This defines an iterative lengthening algorithm using only the component codes. When a $k$--forcing set is also an information set for the Tanner code, we also have an iterative and systematic encoder. We have found that some special subsets of the Grassmannian, are both information sets and $2$--forcing sets for the point--line incidence graph of the Schubert union. This also implies the dual Schubert union codes are generated by their minimum weight codewords. As an application we have also found the minimum distance of Schubert union codes. We have also found lower, upper bounds  and a divisibility condition on the weight of a codeword of $\grassmanncode{\ell}{m}$. We expect that similar techniques will also apply for Schubert codes and the generalized Hamming weights of Grassmann codes.

\section*{Acknowledgements}
The author would like to  thank Sudhir R. Ghorpade and J\o{}rn Justesen for many fruitful discussions. The bulk of this research was carried at the Technical University of Denmark. Part of this work was carried out in the Indian Institute of Technology-- Bombay. The author gratefully acknowledges the support from the Danish National Research Foundation and the National Science Foundation of China (Grant No.11061130539) for the Danish-Chinese Center for Applications of Algebraic Geometry in Coding Theory and Cryptography. The author also thanks the reviewers for their time and advice.


\begin{thebibliography}{10}

\bibitem{ACDP14}
David Amos, Yair Caro, Randy Davila, and Ryan Pepper.
\newblock Upper bounds on the $k$-forcing number of a graph.
\newblock {\em CoRR}, abs/1401.6206, 2014.

\bibitem{BGH12}
Peter Beelen, Sudhir R. Ghorpade and Tom H\o{}holdt.
\newblock Duals of affine {G}rassmann codes and their relatives.
\newblock {\em IEEE trans. info. theory}, vol. 58, issue 6 pages 3843--3855, 2012.

\bibitem{BP15}
Peter Beelen and Fernando Pi\~nero.
\newblock The structure of dual {G}rassmann codes.
\newblock {\em Designs, Codes and Cryptography}, pages 1--20, 2015.

\bibitem{DR09}
Paul~A. Dreyer. and Fred~S. Roberts.
\newblock Irreversible $k$-threshold processes: Graph-theoretical threshold models
  of the spread of disease and of opinion.
\newblock {\em Discrete Applied Mathematics}, 157(7):1615 -- 1627, 2009.

\bibitem{E98}
Jorg Eisfeld.
\newblock Subsets of association schemes corresponding to eigenvectors of the
  {B}ose--{M}esner algebra.
\newblock {\em Bull. Belg. Math. Soc. Simon Stevin}, 5(2/3):265--274, 1998.

\bibitem{GL00}
Sudhir~R. Ghorpade and Gilles Lachaud.
\newblock Higher weights of {G}rassmann codes.
\newblock {\em Coding Theory, Cryptography and Related Areas (Guanajuato,
  1998)}, pages 122--131, 2000.
  
\bibitem{GPP09}
Sudhir ~R. Ghorpade, Arunkumar R. Patil and Harish K. Pillai.
\newblock Decomposable subspaces, linear sections of {G}rassmann varieties and higher weights of {G}rassmann codes.
\newblock {\em Finite fields and their applications }, 15(1):54--68, 2009

\bibitem{GT03}
Sudhir ~R. Ghorpade and Michael A. Tsfasmann
\newblock {S}chubert varieties, linear codes and enumerative combinatorics
\newblock {\em Finite fields and their applications }, 11(4):684--699, 2005


\bibitem{GT05}
Sudhir ~R. Ghorpade and Michael A. Tsfasmann
\newblock Classical varieties, codes and combinatorics, Formal Power Series and Algebraic Combinatorics 
\newblock (Vadstena, 2003) Actes/Proceedings, K. Eriksson and S. Linusson Eds., Link\"oping University, Sweden (2003), pp. 75-84. 

\bibitem{HJR07}
Johan~P. Hansen, Trygve Johnsen, and Kristian Ranestad.
\newblock {S}chubert unions in {G}rassmann varieties.
\newblock {\em Finite Fields and Their Applications}, 13(4):738 -- 750, 2007.

\bibitem{HJR09}
{Johan Peder} Hansen, Trygve Johnsen, and Kristian Ranestad.
\newblock Grassmann codes and Schubert unions.
\newblock {\em Seminaires et Congres}, pages 103--123, 2009.
\newblock Volumne: 21.

\bibitem{JL07}
Heeralal Janwa and A. K. Lal.
\newblock On Generalized Hamming Weights and the Covering Radius of Linear Codes
\newblock Applied Algebra, Algebraic Algorithms and Error-Correcting Codes: 17th International Symposium, AAECC-17, Bangalore, India, December 16-20, 2007. Proceedings pp. 347--356
\newblock Springer Berlin Heidelberg

\bibitem{KL72}
S.L. Kleiman and Dan Laksov.
\newblock {S}chubert calculus.
\newblock {\em The American Mathematical Monthly}, 79(10):1061--1082, December
  1972.

\bibitem{MS77}
F.J. {MacW}illiams and N.J.A. Sloane.
\newblock {\em The theory of error--correcting codes}.
\newblock North Holland, 1977.

\bibitem{N93}
Dimitri~Y. Nogin.
\newblock Codes associated to {G}rassmannians, Arithmetic Geometry and Coding Theory (Luminy 1993), 
R. Pellikaan, M. Perret and S.G. Vl\u{a}dut (Eds.)  Walter de Gruyter, Berlin/New York 1996, pp. 145--154.

\bibitem{P10}
Mark Pankov.
\newblock {\em Grassmannians of Classical Buildings}.
\newblock World Scientific, 2010.

\bibitem{R87a}
C.T. Ryan.
\newblock {An application of {G}rassmannian varieties to coding theory}.
\newblock {\em Congr. Numerantium}, 57:257--271, 1987.

\bibitem{R87b}
C.T. Ryan.
\newblock {Projective codes based on {G}rassmannian varieties}.
\newblock {\em Congr. Numerantium}, 57:273--279, 1987.

\bibitem{S90}
Ernest Shult.
\newblock Geometric hyperplanes of embedabble {G}rassmannians.
\newblock {\em journal of algebra}, 145:55--82, 1992.

\bibitem{S11}
Ernest~E. Shult.
\newblock {\em Points and Lines characterizing the classical geometries}.
\newblock Springer Verlag -- Berlin, 2011.

\bibitem{T81}
R.M. Tanner.
\newblock A recursive approach to low complexity codes.
\newblock {\em {IEEE} transactions on information theory}, 27:533--547, 1981.

\bibitem{NTV07}
Michael A. Tsfasman, Serge G. Vl\u{a}duţ and Dimitri Y. Nogin”.
\newblock {\em Algebraic geometric codes: basic notions}.
\newblock No. 139, American Mathematical Society, 2007.

\bibitem{X08}
Xu~Xiang.
\newblock On the minimum distance conjecture for {S}chubert codes.
\newblock {\em {IEEE} Transactions on Information Theory}, 54(1):486--488,
  2008.

\bibitem{GS16}
Sudhir R. Ghorpade and Prashant Singh
\newblock Minimum Distance and the Minimum Weight Codewords of Schubert Codes
\newblock {arXiV:1609.08265},
\newblock {http://arxiv.org/abs/1609.08265},

\end{thebibliography}

\end{document}